\documentclass{amsart}
\usepackage{amsmath}
\usepackage{algorithm}
\usepackage{algpseudocode}
\usepackage{graphicx}

\newcommand{\ZZ}{\mathbb Z}
\DeclareMathOperator*{\Min}{\mathrm{Min}}

\newtheorem{thm}{Theorem}
\newtheorem{lem}[thm]{Lemma}
\newtheorem{prop}[thm]{Proposition}
\newtheorem{rem}[thm]{Remark}

\theoremstyle{definition}
\newtheorem{defn}[thm]{Definition}

\numberwithin{thm}{section}

\title{
    An algorithm for constructing and classifying the space of small integer weighing matrices}
\author{
    Radel Ben-Av, Giora Dula, Assaf Goldberger, Yoseph Strassler
    }

\begin{document}

\begin{abstract}

    In this paper we describe an algorithm for generating all the possible $PIW(m,n,k)$ - integer $m\times n$ Weighing matrices of weight $k$ up to Hadamard equivalence. Our method is efficient on a personal computer for small size matrices, up to $m\le n=12$, and $k\le 50$. As a by product we also improved the \textit{\textbf{nsoks}} \cite{riel2006nsoks} algorithm to find all possible representations of an integer $k$ as a sum of $n$ integer squares.
    We have implemented our algorithm in \texttt{Sagemath} and as an example we provide a complete classification for \ $n=m=7$ and $k=25$.  Our list of $IW(7,25)$ can serve as a step towards finding the open classical weighing matrix $W(35,25)$.

\end{abstract}

\maketitle

\section{Introduction}

Let us define $PIW(m,n,k) = \{ \ P\ |\ P\in \ZZ^{m\times n} \ ,\ PP^\top =kI\}$ , $\ZZ$ is the ring of integers, and let $IW(n,k)=PIW(n,n,k)$. Classical weighing matrices $W(n,k)$ are the subset of $IW(n,k)$ of matrices over $\{-1,0,1\}$. Weighing matrices $W(n,k)$ have been extensively investigated over the past few decades \cite{MR1137832,Dean:FFA:2019,Dean:AAECC:2021,Flammia:Severini:2009,vanDam:2002,Craigen:JCD:1995:Weaving,Craigen:BICA:1991,Hadi:Thomas:Sho:JCTA:2022,Hadi:Sho:BTR:JACO:2022,M_M_Tan:DCC:2018,Craigen:DCC:1995,KHL:BS:JCTA:2018}. A particular interesting subcase are Hadamard matrices $H(n)=W(n,n)$. Hadamard proved that $H(n)$ is noneempty for $n>2$ only if $n$ is divisible by 4. The Hadamard conjecture is that $H(n)$ is noneempty for all $n$ divisible by $4$. A conjecture on weighing matrices extending the Hadamard conjecture is that \ $W(4l,k)$ is nonempty for every $l$ and every $k\le 4l$. For a (somewhat outdated) summary of methods and existence tables, see \cite[Chap. V]{colbourn2006handbook}.\\

A circulant matrix is square $n\times n$ matrix $C$ such that $C_{i,j}=f(i-j)$ where $f:\ZZ/n \to \mathbb C$ and the indices and operations thereof are taken from the abelian group $(\ZZ/n,+)$. Cirulant weighing matrices, denoted by $CW(n,k)$ exist only if $k$ is a perfect square, and under this assumption, there is a variety of solutions. For example, the weight $k=9$ has been fully classified, see \cite{ANG20082802}. Integral circulant weighing matrices, denoted by $ICW(n,k)$, are a stepping stone for the construction of circulant or general classical weighing matrices, see e.g. \cite{gutman2009circulant}. In more detail, we may try to construct a weiging matrix $\displaystyle M=( C_{ij}) \in IW(\displaystyle n_1n_2,k)$, where the blocks $C_{i,j}$ are circulant of size $n_2\times n_2$. If we replace each circulant block $C_{i,j}$ with its row sum $c_{i,j}=sum(C_{i,j})$, the resulting matrix $S=(c_{i,j})$ is in $IW(n_1,k)$. So it might be easier to first construct $S$, and knowing $S$ will give us a clue as how to construct the matrix $M$. Some of the authors have used this method to construct $W(25,16)$ from a certain $IW(5,16)$ (see \cite{munemasa2017weighing} for the matrix). A special instance of this method applies to the doubling \cite{Seberry2005} and Wiliamson \cite{Williamson1944} constructions. It is worthwhile to mention the concept of ``multilevel Hadamard matrices'' \cite{parker2011multilevel}, which has its separate motivation, and is a special case of $IW(n,k)$.\\


Two $PIW(m,n,k)$ matrices are \emph{Hadamard Equivalent} (or H-equivalent) if one can be obtained from the other using a sequence of the following operations:

\begin{enumerate}
    \item Mutiplying a row or a column by -1, and
    \item Swapping two rows or columns.
\end{enumerate}
Note that $PIW(m,n,k)$ is closed under H-equivalence. In the case of $IW(n,k)$ we can add another operation:

\begin{enumerate}
    \setcounter{enumi}{2}
    \item Transpose the matrix.
\end{enumerate}
Adding this operation we obtain an extension of the previous equivalence relation, which we call \emph{transpose-Hadamard} equivalence relation (or TH-equivalence).\\ 

In this paper we give an algorithm to generate all $PIW(m,n,k)$ up to H-equivalence. Our algorithm works by first generating the first row of the matrix, and then extend each time by adding a new row. This is implemented by a tree whose nodes are partial prefixes of the intended solutions. This is similar to a BFS procedure.
 Keeping track of the H-equivalence class of the partial matrix at each stage, helps to significantly reduce the tree width, which makes it tractable for matrix size up to $12$ and weight up to $50$.\\
 
 To this end we define the \emph{row-lex} ordering. This is a full ordering relation on the set of all integer matrices of a given size. This ordering has some key properties which allow us to efficiently keep track of the H-equivalence class at any stage of the algorithm. For efficiency reasons we add an extra parameter $mindepth$ which controls the running time. This comes at the cost of producing more than one matrix in some H-equivalence classes. To overcome this multiplicity, we use the \emph{code-invariant} (for definition see \S\ref{sec:codeinv} below) to help us separate some non H-equivalent matrices from each other. In the past some of the authors have used that invariant to find a symmetric $W(23,16)$.
 Then we use a procedure (not described in this paper) to establish an explicit H-equivalence between matrices not separated by the code invariant. This gives us the full classification to H-equivalence classes. To proceed to TH-equivalence classification, we transpose the matrices in that list, compute again the code-invariants, and prove H-equivalences if needed.\\

An important building block is the \texttt{NSOKS} algorithm to find all representations of an integer $k$ as a sum of $r$ integer squares. An implementation of this already exists \cite{riel2006nsoks} and used in \cite{bright2016mathcheck} and other papers to represent integers as sums of four squares. Our needs are more extensive, and below (\S\ref{sec:NSOKS}) we give an improved version.\\

This entire procedure has been implemented for $IW(7,25)$, and a the full list of all $44$ TH-inequivalent solutions is given.

\section{The \texttt{NSOKS} algorithm}\label{sec:NSOKS}

The \texttt{NSOKS} algorithm computes the collection of all representations of a positive integer $n$ as a sum of $r$ nonnagative squares. The input is the number $n$, an integer $r=$ the number of required squares and an optional argument $maxsq$ which is an upper bound for the integers $s$ in the representation. The output is a list $[S_1,\ldots,S_t]$, where each $S_i$ is a list $[(s_1,m_1),\ldots, (s_l,m_l)]$ with $s_1<s_2<\ldots s_l\le maxsq$ such that $\sum m_i=r$ and $\sum m_is_i^2=n$. Each $S_i$ is a representation of $n$ as a sum of $r$ nonnegative squares, and $S$ is the full list of all possible such representations, up to ordering the squares.\\

A Maple implementation \cite{riel2006nsoks} exists on the web. Nevertheless, our SageMath implementation runs faster. For example, our \texttt{NSOKS}$(200,200)$ has $27482$ representations and runs on our SageMath machine in 0.3s. In comparison, the Maple code adapted to SageMath, on our machine, runs in 13s. Both codes have been checked to give the same answer. Our algorithm advances by recursion, from the largest square down to zero. The algorithm loops on the largest square $s^2$ and its multiplicity $m_s$. Then we descend to $n\to n-m_ss^2$, and call \texttt{NSOKS} by recursion, this time by setting $maxsq=s-1$. The main point of improvement over \cite{riel2006nsoks} is that we work with multiplicities, thus reducing the recursion depth. The second point is that once we get down to $maxsq=1$, we do not need to recurse any more, and the answer is determined immediately (line \ref{step:1}).

\begin{algorithm}[h]
    \caption{Find all representations of $n$ as a sum of $r$ nonnegative squares }\label{alg:nsoks}
    \begin{algorithmic}[1]
    \Procedure{\texttt{NSOKS}}{$n,r,maxsq=False$}
    \If{$maxsq=1$} \textbf{return} $[[(1,n),(0,r-n)]]$ \Comment{{\scriptsize No need for recursion\label{step:1}.}}
    \EndIf
    \State $M\gets \lfloor \sqrt{n} \rfloor$
    \If{$maxsq$}
        \State $M\gets \min(maxsq,M)$
    \EndIf
    \State $L\gets \lceil \sqrt{n/r}\rceil$
    \State $SquaresList\gets []$
    \For{$s\in [L,M]$} \Comment{{\scriptsize Loop on the square.}}
        \For{$i\in [1,\lfloor n/s^2\rfloor]$} \Comment{{\scriptsize Loop on the multiplicity.}}
            \State $n'\gets n-i\cdot s^2$
            \If{$i=r$}
                \State Append $[(s,r),]$ to $SquaresList$.
            \Else 
                \State $rem\gets$ \texttt{NSOKS}$(n',r-i,maxsq=s-1)$ \Comment{{\scriptsize The recursion step.}}
                \For{$SubSquaresList \in rem$}
                    \State Append $[(s,i),*SubSquaresList]$ to $SquareList$
                \EndFor
            \EndIf

        \EndFor
    \EndFor
    \State \textbf{return} SquaresList
    \EndProcedure
    \end{algorithmic}
\end{algorithm}

 \section{The row-lex ordering and the search algorithm}
 In this section we define the row-lex ordering on the set of integer matrices of a given size $m \times n$, and prove some basic properties of this ordering. Using these properties, we design a search algorithm to find an exhaustive list of all $PIW(m,n,k)$ up to Hadamard equivalence. The output of our algorithm may contain more than one candidate in a single Hadamard class, and in the next section we will discuss a post processing procedure towards a correction of this flaw.\\
 
 \subsection{The row-lex ordering} The discussion here is not limited to (partial) weighing matrices, and we consider all integer matrices of a given size $p\times n$. We denote this set by $\ZZ^{m\times n}$. The set $\ZZ$ of all integers carries its natural ordering $\le$. We extend first this ordering to the set of $n$-vectors, $\ZZ^n$ by the lexicographic extension of $\le$, still denoted $\le$. This means that
 \begin{multline*}
    (v_1,\ldots,v_n)<(w_1,\ldots,w_n), \ \text{iff }\exists j  \ (v_1,\ldots,v_{j-1})=(w_1, \ldots, w_{j-1}) \text{ and } v_j<w_j.
\end{multline*}
Next we extend this ordering to $m\times n$-matrices by lexicographic extension of $\le$ on the rows (i.e. viewing the matrix as a vector of rows). We write this ordering as $M\le_R N$. This is called the \emph{row-lex ordering}. Similarly we can consider the \emph{column-lex ordering}, by extending $\le$ on columns, viewing the matrix as a vector of columns. We shall write $M\le_C N$ for this ordering. The two orderings are not equal, and for our algorithm which looks at situations where $m\le n$ it will be more appropriate to use the row-lex ordering.\\

Some notation is in order: For any matrix $M$, let $M_i$ denote its $i$th row and let $M^j$ denote its $j$th column. Let $M_{i:k}$ denote the submatrix whose rows are $M_i,M_{i+1},\ldots,M_{k-1}$ given in this order. We denote $M^{j:l}$ analogously for columns. Let $(-1)_iM$ denote the matrix $M$ with $M_i$ replaced by $-M_i$. More generally we denote $(-1)_SM$, for a set of indices $S$, as the matrix $M$ with $M_i$ replaced by $-M_i$ for all $i\in S$.
Similarly we denote $(-1)^jM$ and $(-1)^SM$ for columns. For each matrix $M$, let $[M]$ denote its Hadamard equivalence class. Let 
$$\Min(M) = \min \{A\ | \ A\in [M]\},$$ the minimum is taken with respect to the row-lex ordering. We say that $M$ is \emph{minimal} if $M=\Min(M)$. In each Hadamard class there exists a unique minimal matrix. We now study some proprties of minimal matrices. We say that a vector $v$ \emph{begins with a positive (resp. negative) entry} if for some $j$, $v_1=\cdots=v_{j-1}=0$ and $v_j>0$ (resp. $v_j<0$).

\begin{lem}
    In a minimal matrix each nonzero row and each nonzero column begins with a negative entry.
\end{lem}

\begin{proof} Let $M$ be minimal.
    Suppose that a row $M_i$ begins with a positive entry. Then $-M_i<M_i$, and by definition $(-1)_i M<M$, in contradiction to minimality.\\
    Now suppose that $M^j$ begins with a positive entry, sitting at the position $(i,j)$. Then in $(-1)^jM$, the first $i-1$ rows remain unchanged, while $((-1)^jM)_i<M_i$, which in turn implies that $(-1)^jM<M$, again contradicting the minimality of $M$.
\end{proof}

One consequence of the proof, that we shall not use in this paper, is the following statement.

\begin{thm}
    Each matrix $M$ can be brought, using only row and column negations, to a form where each nonzero row and colum begins with a negative entry. 
\end{thm}

\begin{proof}
    Sequences of row and column negations define an equivalence relation on matrices. The proof of the above lemma shows that the minimal representative in a class has the desired property. 
\end{proof}

\begin{lem}
    In a minimal matrix $M\in \ZZ^{m\times n}$, the \underline{columns} are in increasing order: $M^1\le M^2\le \cdots \le M^n$.
\end{lem}

\begin{proof}
    Suppose by contradiction that $M^{j-1}>M^j$. Let $j$ be smallest one with this property. Then for some $i$, $M_{s,j-1}=M_{s,j}$ for all $s<i$ and $M_{i,j-1}>M_{i,j}$. By swapping columns $j,j-1$ we obtain a matrix $M'$ in which rows $1,2,\ldots,i-1$ did not change, while row $i$ has decreased. Thus $M'<_R M$, a contradiction.
\end{proof}

The following is a key property in our algorithm.

\begin{prop} \label{part_min}
    For the row-lex ordering, a matrix $M$ is minimal, if and only if for all $i$, $M_{1:i}$ is minimal.
\end{prop}

\begin{proof}
    Clearly if $M_{1:i}<_RM'_{1:i}$ then $M<_RM'$. If $M_{1:i}$ is not minimal, then we can perform Hadamard operations on $M$ involving all columns and only the first $i$ rows, to decrease $M_{1:i}$. The resulting matrix $M'<_R M$, in contradiction to the minimality of $M$.
\end{proof}
We remark that in general the initial column submatrices $M^{1:j}$ of a minimal matrix $M$ need not be minimal. Our algorithm will build matrices row by row, and this explains why we prefer to use the row-lex ordering rather than the column-lex counterpart.

\subsection{Minimizing a class}
In this short section we describe the algorithm \texttt{MINCLASS} to find the minimal representative in a Hadamard class.
Suppose that we are given a matrix $M\in \ZZ^{m\times n}$. Let $Mon(m)$ denote the set of all monomial $m\times m$ matrices with values in $\{0,-1,1\}$. Let $Neg(M)$ denote the matrix obtained from $M$ by negating each column that begins with a positive entry. Let $Ord(M)$ be the matrix obtained from $M$ by permuting its columns to be written from left to right in increasing column order. Consider the following algorithm:

\begin{algorithm}
    \caption{Minimizing a Hadamard class}\label{minclass}
    \begin{algorithmic}[1]
    \Procedure{\texttt{MINCLASS}}{$M$}
    \State $m\gets $ height, $n\gets$ width
    \State $Min\gets M$
        \For{$P\in Mon(m)$}\Comment{go over all row negations and permutations}
        \State $N\gets PM$
        \State $N\gets Neg(N)$  
        \State $N\gets Ord(N)$ 
        \If{$N<_R Min$}
            \State $Min\gets N$
        \EndIf
        \EndFor
    \State \textbf{return} $Min$
    \EndProcedure
    \end{algorithmic}
\end{algorithm}

\begin{prop}
The procedure \texttt{\textsc{MINCLASS}}$(M)$ returns the minimal matrix in the class of $M$.  
\end{prop}

\begin{proof}
    Let $M_0=PMQ$ be the mininal representative in the Hadamard class of $M$, $P,Q$ are monomial. The algorithm
 enumerates over $P\in Mon(m)$ and for the correct $P$ we have $N:=PM=M_0Q^{-1}$. It suffices to show that $M_0=Ord(Neg(N))$. The nonzero columns of $Neg(N)$ and of $M_0$ all begin with a negative entry, so both matrices have the same multiset of columns, which means that $M_0=Neg(N)\Pi$ for a permutation matrix $\Pi$. Since the columns of $M_0$ are in increasing order, then necessarily $Ord(Neg(N))=M_0$.
\end{proof}

\subsection{The main search algorithm}
Now we turn to the main algorithm \texttt{RepPIW} which outputs a list of representatives of all Hadamard classes in $PIW(m,n,k)$. In its default implementation the program outputs exactly one matrix per class, however it contains an optional parameter, `mindepth', which can improve the running time, at the cost of listing one or more matrices per a single class. Before stating the algorithm we give a concise description.\\

The algorithm relies on Proposition \ref{part_min} that initial submatrices of a minimal matrix are minimal. 
The starting point is a list of all minimal integral vectors of weigth $k$, which is in bijection with the output of \texttt{NSOKS}$(n,k)$. This gives the list for $PIW(1,n,k)$. At each stage the algorithm holds a list $MinPIW(p,n,k)$ of all minimal representatives of the $PIW(p,n,k)$. To each member $X_p\in MinPIW(p,n,k)$, 
we produce the list $LV(X_p)$ of all integral vectors of weight $k$ that are (i) larger than the last row of $X_p$, and (ii) are orthogonal to all rows of $X_p$. Then for each $v\in LV(X_p)$ we obtain the matrix $X_{p+1}=[X_p,v]$ by adding a new row below $X_p$. Using \texttt{MINCLASS}, we test if $X_{p+1}$ is minimal. We add it to the new list $MinPIW(p+1,n,k)$ iff it is minimal. Stopping at $p=m$, Proposition \ref{part_min} gurantees that we have correctly created a list representatives for all Hadamard classes of $PIW(m,n,k)$.\\

One improvement that we add, which greatly affects the performence is the parameter `mindepth' which tells the algorithm to stop using \texttt{MINCLASS} if $p>mindepth$. When $p$ is greater we just add any vector $v$ satisfying (i) and (ii). The assumption here is that there are not too many vectors left, and thus the final list is not too large. On the positive side we save a lot of time of minimizing. We will potentially get more representatives than necessary, but as will be discussed below, we have an effective way to tell which are isomorphic to which, eventually yielding the list we want. Following is the pseudo-code. In this algorithm we use the following notation: `SignedPerms$(v)$' is the set of all permutations and element negations of a vector $v$. For a matrix $M$, recall that $M_i$ denotes its $i$th row. Let $M^-$ denote the matrix without its last row. Let $[M,v]$ denote the matrix $M$, augmented by the additional row $v$.\\

\begin{algorithm}
    \caption{Generating Hadamard representatives of $PIW(m,n,k)$}\label{RepPIW}
    \begin{algorithmic}[1]
    \Procedure{\texttt{RepPIW}}{$m,n,k,mindepth=m$}
    \State $SOKS\gets$ \texttt{NSOKS}$(n,k)$
    \State $MinPIW[1]\gets $ [\texttt{MINCLASS}$(v)$ for $v$ in $SOKS$] 
    \State $AllRows\gets$ SignedPerms$(SOKS)$ \Comment{{\scriptsize This is the full resrvoir of all possible rows.}} 
    \If{m=1}
        \State \textbf{return} $MinPIW[1]$
    \EndIf
    \For{$v \in MinPIW[1]$} \Comment{{\scriptsize Compute a list of all second rows.}}
            \State $R(v)\gets $[$w\in AllRows$ if $wv^\top=0$ \& $w>v$]
    \EndFor
    \For{$p=1$ to $m-1$}
        \For{$X\in MinPIW[p]$}
            \State $R(X)\gets $ [$w\in R(X^-)$ if $X_p w^\top=0$ \& $w>X_p$] \ \ \\ \Comment{Make a list of all $p$th rows. This tests conditions (i) and (ii)}
            \For{$w\in R(X)$}
                \State $X_{new}=[X,w]$
                \If{$p\le mindepth$}
                    \If{$X_{new}==$ \texttt{MINCLASS}$(X_{new})$}
                        \State Append $X_{new}$ to $MinPIW[p+1]$.
                    \EndIf
                \Else
                    \State Append $X_{new}$ to $MinPIW[p+1]$.
                \EndIf
            \EndFor
        \EndFor
    \EndFor
    \State \textbf{return} $MinPIW[m]$
    \EndProcedure
    \end{algorithmic}
\end{algorithm}

\begin{thm}
    The function \textsc{\texttt{RepPIW}}$(m,n,k)$ outputs the list $MPIW(m,n,k)$ of all minimal of Hadamard representatives of $PIW(m,n,k)$.\\ The function 
    \textsc{\texttt{RepPIW}}$(m,n,k,mindepth=d)$ outputs a larger list of $PIW(m,n,k)$ containing all minimal elements.
\end{thm}

\begin{proof}
    The proof of the first part is by induction on $m$. For $m=1$ this is clear, as $MPIW(1,n,k)$ is the list of all minimal vector is $SOKS$. Assuming validity for $m-1$, we enter the for loop at $p=m-1$ (line 11) with $MinPIW(m-1)=MPIW(m-1,n,k)$ by the induction hypothesis. Suppose that $M\in MPIW(m,n,k)$. Then by Theorem \ref{part_min} $M_{1:m-1}\in MinPIW(m-1)$. The list $R(M_{1:m-1})$ holds all vectors that are orthogonal to $M_{1:m-1}$. Thus the vector $M_m$ enters the list $R(M_{1:m-1})$ (line 13) and passes the minimality test (line 17), allowing $M$ to enter the list $MinPIW(m)$ (line 18). This proves that $MinPIW(m)\supseteq MPIW(m,n,k)$. The opposite inclusion is clear as line $18$ allows only minimal matrices. This proves the first assertion. The second assertion follows easily, as we do not always perform the minimality test, but yet the minimal matrices pass all tests. 
\end{proof}

\subsection{Imporving \texttt{MINCLASS}}
The procedure \texttt{MINCLASS}$(M)$ becomes impracticle as the number of rows $m$ becomes large, since we have a factor of $2^mm!$ which is the size of $Mon(M)$. We suggest an improvement which can greatly reduce complexity, however, so far we have not implemented this, and we are not able to estimate the worst case complexity. It looks like the `average' case complexity is low (again we find it difficult to define what is `average').\\

The idea is simple. We first minimize the indivdual rows of $M\in \ZZ^{m\times n}$. Only the smallest row(s) can be selcted as the first row of $Min(M)$. Having chosen the first row, we now adjoin all remaining vectors as candidates to the second row of $Min(M)$. Then we minimize the resulting $2\times n$ matrices. Again we only choose the smallest one(s). We proceed similarly with the third row(s) and so on. Crucially, note that for the minimization we do not need to go over $Mon(p)$. We only need to add the new row and its negation, and then minimize by columns. This minimization will not alter the first $p-1$ rows (as they form a minimal matrix).
Also note that our choice of the first $p-1$ rows is the smallest possible, which is necessary for them to be the matrix $Min(M)_{1:p-1}$. Below is the pseudocode.\\

\begin{algorithm}
    \caption{Fast Minimizing a Hadamard class}\label{fastminclass}
    \begin{algorithmic}[1]
    \Procedure{\texttt{FastMINCLASS}}{$M,Init=(\ ),RIndList=\{[1,2\ldots,m],\}$}
    \State $m\gets $ height, $n\gets$ width.
    \If{height(Init)==m} \Comment{{\scriptsize This is if $Init$ is the full matrix.}}
        \State \textbf{return} Init 
    \EndIf
    \For{$Inds\in RIndList$} \Comment{{\scriptsize This is a list of unused row numbers.}}
        \State $Ns\gets [\ ]$
        \For{$i\in Inds$}
            \State $v\gets M_i$
            \State $N_1\gets [Init,v]$
            \State $N_2\gets [Init,-v]$ \Comment{{\scriptsize We test if $v$ or $-v$ is to be added.}}
            \State $N[i]\gets \min(Ord(Neg(N_1)),Ord(Neg(N_2)))$.
            \State Append $N[i]$ to $Ns$
        \EndFor
        \State $NewInit[Inds]\gets \min(Ns)$.
        \State $NewRIndList[Inds]\gets [Inds \setminus \{i\} \text{for all $i$ if} N[i] == \min(Ns)]$\\ \Comment{{\scriptsize Remove index $i$ if this gave a minimum.}}
    \EndFor
    \State $MinN\gets \min(NewInit[Inds] \text{ for all } Inds)$ \Comment{Pick the absolute minimum.}
    \State $MIndList\gets \{NewRIndList[Ind] \text{ for all } Inds \text{ if } NewInit[Inds]=MinN\}$
    \State $MinM \gets$\texttt{FastMINCLASS}$(M,MinN,MIndList)$ \Comment{{\scriptsize Recursion on new initials.}}
    \State \textbf{return} $MinM$ 
    \EndProcedure
    \end{algorithmic}
\end{algorithm}

In this algorithm the input is a matrix $M$, an initial matrix $Init$ which is supposed to be $Min(M)_{1:p}$ and a set $RIndList$ of lists of indices, where each list containes the row indices not used in $Init$ (there might be few options due to branching). The algorithm constructs the minimal matrix in the class of $M$ subject to the constraint that its $1:p$ part equals $Init$.\\ 

\begin{rem}
    If there is no branching, i.e. there is just one candidate added to an initial at each time, the algorithm finishes quickly. Otherwise we will suffer from branching. There are cases with vast branching, such as scalar matrices, but we feel that on `average' there will be only small branching. We find it hard to estimate the average effect.
\end{rem}

\begin{rem}
    Some of the branching is caused by matrix automorphisms (i.e. Hadamard self equivalences). If this were the only cause, we could just settle for a greedy algorithm, picking up the first candidate row each time, thus avoid branching. We know however, that in classical Hadamard matrices every 3-tuple of rows minimizes to the same matrix giving way to massive branching, regradless of automorphisms.
\end{rem}
    
\section{Results for $IW(7,25)$}
In this section we report on the performence of our algorithm to classify $IW(7,25)$ up to Hadamard equivalence. We ran the main algorithm \texttt{RepPIW}$(7,7,25,mindepth=4)$. The implementation was programmed on SageMath \cite{sagemath} on a Dell laptop with Core i7 and 8GB ram. The running time was $2$ miuntes. The output was a list of $420$ matrices in $IW(7,25)$.\\

Next we have computed the \emph{code invariant} on each matrix, which we now define.

\subsection{The code invariant}\label{sec:codeinv}
For an integral matrix $D\in [-L,L]^{d\times n}$, we compte the value $Code(M)$ which is the vector $\mathbf{b}D$, where $\mathbf{b}=[b^{d-1},\ldots,b^2,b,1]$, $b=2L+1$. The correspondence $D\to Code(D)$ is an injection. We write $D\prec_{d}M$ if $D$ is a $d\times n$ submatrix of $M$. For any matrix $M\in [-L,L]^{r\times n}$, we define the \emph{code invariant} to be 
$$ CodeInv(M,d)\ := \ \text{Multiset}\{Code(Min(D))\ | \ D\prec_d M   \}.
$$
This is clearly a Hadamard invariant of $M$.\\

We have computed $CodeInv(M,3)$ for all $M$ in our list of $420$ matrices, and discovered that this invariant breaks our list into $49$ sublists, $L_1,\ldots,L_{49}$. The members of each list $L_j$ have the same code invariant, and the memebers of different lists have different code invariants.\\

At this point we have verified that in each list $L_j$, all elements are Hadmard equivalent, by producing the monomial transformations. 
Finally, we have reduced our list to $44$ elements, each pair of them is not Hadamard equivalent, nor equivalent to the transpose. To summarize,

\begin{thm}
    Up to Hadamard equivalence, there are $49$ nonequivalent matrices in $IW(7,25)$. Allowing transposition, these reduce to only $44$ of classes.
\end{thm}

\subsection{All $IW(7,25)$ up to TH equivalence.}

\begin{defn}
    A matrix $M\in IW(n,k)$ is \emph{imprimitive} if it is H-equivalent to a block  sum of smaller $IW(n_i,k)$. Otherwise we say that $M$ is \emph{primitive}. 
\end{defn}

As it turns out, exactly $19$ matrices of our list of $44$ are primitive. The rest are H equivalent to block sums of $IW(1,25),IW(2,25),IW(4,25),IW(5,25)$ and $IW(6,25)$. We shall use the notation $n_1A_1\oplus n_2A_2\cdots \oplus n_rA_r$ to denote a block sum of $n_1$ copies of $A_1$, $n_2$ copies of $A_2$ and so on. We first list the primitive $IW(r,25)$ for $r\le 7$.\\

\begin{itemize}
    \item[{\large \bf --$IW(1,25)$:}] $$A_1=[5].$$
    \item[{\large \bf --$IW(2,25)$:}] $$B_1=\begin{bmatrix}
        3 & 4 \\ 4 & -3
    \end{bmatrix}.$$
    \item[{\large\bf --$IW(3,25)$:}] $$\emptyset.$$
    \item[{\large \bf --$IW(4,25)$:}] $$C_1=\left[\begin{array}{rr|rr}
        1 & 4 & 2 & -2 \\
        4 & 1 & -2 & 2 \\
        \hline
        2 & -2 & 4 & 1 \\
        -2 & 2 & 1 & 4
        \end{array}\right], C_2=\left[\begin{array}{rr|rr}
            1 & 4 & 2 & -2 \\
            4 & -1 & -2 & -2 \\
            \hline
            2 & -2 & 4 & 1 \\
            2 & 2 & -1 & 4
            \end{array}\right]$$
    \item[{\large \bf --$IW(5,25)$:}]
            $$D_1=\left[\begin{array}{r|rrrr}
                3 & -2 & -2 & -2 & -2 \\
                \hline
                2 & -2 & 0 & 1 & 4 \\
                2 & 4 & -2 & 0 & 1 \\
                2 & 1 & 4 & -2 & 0 \\
                2 & 0 & 1 & 4 & -2
                \end{array}\right], D_2=\left[\begin{array}{r|rrrr}
                    3 & 2 & 2 & 2 & 2 \\
                    \hline
                    2 & 3 & -2 & -2 & -2 \\
                    2 & -2 & 3 & -2 & -2 \\
                    2 & -2 & -2 & 3 & -2 \\
                    2 & -2 & -2 & -2 & 3
                    \end{array}\right]$$

    \item[{\large \bf --$IW(6,25)$:}] 
    $E_i=$\\[0.3cm]
    \scalebox{0.5}[0.5]{$%
    \left[\begin{array}{rrr|rrr}
        -4 & 1 & 0 & 2 & 0 & 2 \\
        0 & -4 & 1 & 2 & 2 & 0 \\
        1 & 0 & -4 & 0 & 2 & 2 \\
        \hline
         -2 & -2 & 0 & -4 & 0 & 1 \\
        0 & -2 & -2 & 1 & -4 & 0 \\
        -2 & 0 & -2 & 0 & 1 & -4
        \end{array}\right], \left[\begin{array}{rrrrrr}
            4 & 2 & 2 & 1 & 0 & 0 \\
            2 & 1 & -4 & -2 & 0 & 0 \\
            2 & -4 & 0 & 0 & 2 & 1 \\
            0 & 0 & 2 & -4 & 1 & -2 \\
            0 & 0 & 1 & -2 & -2 & 4 \\
            -1 & 2 & 0 & 0 & 4 & 2
            \end{array}\right],
         \left[\begin{array}{rrrrrr}
                4 & 2 & 2 & 1 & 0 & 0 \\
                2 & 0 & -3 & -2 & 2 & 2 \\
                2 & -2 & -2 & 0 & -2 & -3 \\
                1 & -4 & 2 & 0 & 0 & 2 \\
                0 & 0 & 2 & -4 & 1 & -2 \\
                0 & -1 & 0 & 2 & 4 & -2
                \end{array}\right] ,
                \left[\begin{array}{rrrrrr}
                    4 & 2 & 2 & 1 & 0 & 0 \\
                    2 & 0 & -3 & -2 & 2 & 2 \\
                    2 & -3 & 0 & -2 & -2 & -2 \\
                    1 & -2 & -2 & 4 & 0 & 0 \\
                    0 & 2 & -2 & 0 & 1 & -4 \\
                    0 & 2 & -2 & 0 & -4 & 1
                    \end{array}\right] , 
                   $}\\[0.3cm]

                \scalebox{0.5}[0.5]{$ 
                 \left[\begin{array}{rrrrrr}
                4 & 2 & 2 & 1 & 0 & 0 \\
                2 & 0 & -4 & 0 & 2 & 1 \\
                2 & -4 & 0 & 0 & -1 & -2 \\
                1 & 0 & 0 & -4 & -2 & 2 \\
                0 & 2 & -1 & -2 & 0 & -4 \\
                0 & 1 & -2 & 2 & -4 & 0
                \end{array}\right] ,  
                \left[\begin{array}{rrrrrr}
                3 & 3 & 2 & 1 & 1 & 1 \\
                3 & -1 & -1 & 1 & -2 & -3 \\
                2 & -1 & -3 & -1 & 3 & 1 \\
                1 & 1 & -1 & -3 & -3 & 2 \\
                1 & -2 & 3 & -3 & 1 & -1 \\
                1 & -3 & 1 & 2 & -1 & 3
                \end{array}\right] , 
                 \left[\begin{array}{rrrrrr}
                3 & 3 & 2 & 1 & 1 & 1 \\
                3 & -1 & -1 & 1 & -2 & -3 \\
                2 & -3 & 1 & -1 & -1 & 3 \\
                1 & 2 & -3 & -3 & -1 & 1 \\
                1 & -1 & 1 & -3 & 3 & -2 \\
                1 & -1 & -3 & 2 & 3 & 1
                \end{array}\right] , 
                 \left[\begin{array}{rrrrrr}
                3 & 3 & 2 & 1 & 1 & 1 \\
                3 & -2 & 1 & -1 & -1 & -3 \\
                2 & -3 & -1 & 1 & 1 & 3 \\
                1 & 1 & -1 & -3 & -3 & 2 \\
                1 & 1 & -3 & 3 & -2 & -1 \\
                -1 & -1 & 3 & 2 & -3 & 1
                \end{array}\right] , $}\\[0.5cm]
                \scalebox{0.5}[0.5]{$
                \left[\begin{array}{rrrrrr}
                3 & 3 & 2 & 1 & 1 & 1 \\
                3 & -2 & -3 & 1 & 1 & 1 \\
                1 & 1 & -1 & 2 & -3 & -3 \\
                2 & -3 & 3 & -1 & -1 & -1 \\
                1 & 1 & -1 & -3 & 2 & -3 \\
                -1 & -1 & 1 & 3 & 3 & -2
                \end{array}\right] , 
                \left[\begin{array}{rrrrrr}
                3 & 3 & 2 & 1 & 1 & 1 \\
                3 & -3 & 1 & 1 & -1 & -2 \\
                2 & 1 & -3 & -1 & -3 & 1 \\
                1 & 1 & -1 & -3 & 2 & -3 \\
                1 & -1 & -3 & 2 & 3 & 1 \\
                -1 & 2 & -1 & 3 & -1 & -3
                \end{array}\right] , 
                \left[\begin{array}{rrrrrr}
                3 & 3 & 2 & 1 & 1 & 1 \\
                3 & -3 & 1 & 1 & -1 & -2 \\
                2 & -1 & -1 & -3 & -1 & 3 \\
                1 & 2 & -3 & 1 & -3 & -1 \\
                1 & 1 & -1 & -3 & 2 & -3 \\
                1 & -1 & -3 & 2 & 3 & 1
                \end{array}\right] , 
                \left[\begin{array}{rrrrrr}
                3 & 2 & 2 & 2 & 2 & 0 \\
                2 & 2 & 0 & -2 & -3 & 2 \\
                2 & 0 & -2 & -3 & 2 & -2 \\
                2 & -2 & -3 & 2 & 0 & 2 \\
                2 & -3 & 2 & 0 & -2 & -2 \\
                0 & -2 & 2 & -2 & 2 & 3
                \end{array}\right],$}\\[0.3cm]
                \scalebox{0.5}[0.5]{$
                \left[\begin{array}{rrrrrr}
                    4 & 2 & 2 & 1 & 0 & 0 \\
                    2 & -2 & -2 & 0 & 3 & 2 \\
                    2 & -2 & -2 & 0 & -2 & -3 \\
                    1 & 0 & 0 & -4 & -2 & 2 \\
                    0 & 3 & -2 & -2 & 2 & -2 \\
                    0 & 2 & -3 & 2 & -2 & 2
                    \end{array}\right]
    $}\\[0.3cm]

\item[{\large \bf --$IW(7,25)$:}] $F_i=$\\[0.1cm]
\scalebox{0.45}[0.45]{$
        \left[\begin{array}{rrrrrrr}
        4 & 2 & 2 & 1 & 0 & 0 & 0 \\
        2 & 0 & -3 & -2 & 2 & 2 & 0 \\
        2 & -4 & 0 & 0 & 0 & -2 & 1 \\
        1 & 0 & -2 & 0 & -4 & 0 & -2 \\
        0 & 2 & -2 & 0 & 1 & -4 & 0 \\
        0 & 1 & 0 & -2 & -2 & 0 & 4 \\
        0 & 0 & 2 & -4 & 0 & -1 & -2
        \end{array}\right] , 
        \left[\begin{array}{rrrrrrr}
        4 & 2 & 2 & 1 & 0 & 0 & 0 \\
        2 & -1 & -1 & -4 & 1 & 1 & 1 \\
        1 & 1 & -3 & 0 & 1 & -2 & -3 \\
        1 & 0 & -3 & 2 & -1 & 3 & 1 \\
        1 & -1 & -1 & 0 & -3 & -3 & 2 \\
        1 & -3 & 1 & 0 & -2 & 1 & -3 \\
        1 & -3 & 0 & 2 & 3 & -1 & 1
        \end{array}\right] , 
        \left[\begin{array}{r|rrr|rrr}
            4 & 2 & 2 & 1 & 0 & 0 & 0 \\
            \hline
             1 & -4 & 1 & 2 & 1 & 1 & 1 \\
            2 & -1 & -1 & -4 & 1 & 1 & 1 \\
            1 & 1 & -4 & 2 & 1 & 1 & 1 \\
            \hline
             1 & -1 & -1 & 0 & -3 & -3 & 2 \\
            1 & -1 & -1 & 0 & -3 & 2 & -3 \\
            1 & -1 & -1 & 0 & 2 & -3 & -3
            \end{array}\right] , 
        \left[\begin{array}{rrrrrrr}
        4 & 2 & 2 & 1 & 0 & 0 & 0 \\
        2 & -1 & -2 & -2 & 2 & 2 & 2 \\
        2 & -2 & -1 & -2 & -2 & -2 & -2 \\
        1 & -2 & -2 & 4 & 0 & 0 & 0 \\
        0 & 2 & -2 & 0 & 3 & -2 & -2 \\
        0 & 2 & -2 & 0 & -2 & 3 & -2 \\
        0 & 2 & -2 & 0 & -2 & -2 & 3
        \end{array}\right] , $}\\[0.3cm]
        \scalebox{0.45}[0.45]{$
        \left[\begin{array}{rrrrrrr}
        4 & 2 & 2 & 1 & 0 & 0 & 0 \\
        2 & -1 & -2 & -2 & 2 & 2 & 2 \\
        2 & -2 & -2 & 0 & 0 & -2 & -3 \\
        0 & 2 & 0 & -4 & -1 & -2 & 0 \\
        1 & -2 & 0 & 0 & -4 & 0 & 2 \\
        0 & 2 & -2 & 0 & -2 & 3 & -2 \\
        0 & 2 & -3 & 2 & 0 & -2 & 2
        \end{array}\right] , 
        \left[\begin{array}{rrrrrrr}
        4 & 2 & 2 & 1 & 0 & 0 & 0 \\
        2 & -1 & -3 & 0 & 3 & 1 & 1 \\
        1 & 1 & -1 & -4 & -1 & 1 & -2 \\
        1 & 0 & -3 & 2 & -3 & -1 & -1 \\
        1 & -1 & 0 & -2 & -1 & -3 & 3 \\
        1 & -3 & 1 & 0 & 1 & -2 & -3 \\
        1 & -3 & 1 & 0 & -2 & 3 & 1
        \end{array}\right] , 
        \left[\begin{array}{rrrrrrr}
        4 & 2 & 2 & 1 & 0 & 0 & 0 \\
        2 & -1 & -3 & 0 & 3 & 1 & 1 \\
        1 & 1 & -1 & -4 & -2 & 1 & 1 \\
        1 & 0 & -3 & 2 & -3 & -1 & -1 \\
        1 & -1 & 0 & -2 & 1 & -3 & -3 \\
        1 & -3 & 1 & 0 & -1 & 3 & -2 \\
        1 & -3 & 1 & 0 & -1 & -2 & 3
        \end{array}\right] , 
        \left[\begin{array}{rrrrrrr}
        4 & 2 & 2 & 1 & 0 & 0 & 0 \\
        2 & -1 & -3 & 0 & 3 & 1 & 1 \\
        1 & 1 & -3 & 0 & -3 & 1 & -2 \\
        1 & -1 & 1 & -4 & -1 & 2 & 1 \\
        1 & -1 & 0 & -2 & 1 & -3 & -3 \\
        1 & -1 & -1 & 0 & -2 & -3 & 3 \\
        1 & -4 & 1 & 2 & -1 & 1 & -1
        \end{array}\right] , $}\\[0.3cm]
        \scalebox{0.45}[0.45]{$
        \left[\begin{array}{rrrrrrr}
        4 & 2 & 1 & 1 & 1 & 1 & 1 \\
        2 & 1 & -2 & -2 & -2 & -2 & -2 \\
        1 & -2 & 4 & -1 & -1 & -1 & -1 \\
        1 & -2 & -1 & 4 & -1 & -1 & -1 \\
        1 & -2 & -1 & -1 & 4 & -1 & -1 \\
        1 & -2 & -1 & -1 & -1 & 4 & -1 \\
        1 & -2 & -1 & -1 & -1 & -1 & 4
        \end{array}\right] , 
        \left[\begin{array}{rrrrrrr}
        4 & 2 & 1 & 1 & 1 & 1 & 1 \\
        2 & 1 & -2 & -2 & -2 & -2 & -2 \\
        1 & -2 & 4 & -1 & -1 & -1 & -1 \\
        1 & -2 & -1 & 3 & 1 & 0 & -3 \\
        1 & -2 & -1 & 1 & 0 & -3 & 3 \\
        1 & -2 & -1 & 0 & -3 & 3 & 1 \\
        1 & -2 & -1 & -3 & 3 & 1 & 0
        \end{array}\right] , 
        \left[\begin{array}{rrrrrrr}
        4 & 2 & 1 & 1 & 1 & 1 & 1 \\
        2 & -1 & 2 & -2 & -2 & -2 & -2 \\
        1 & 2 & -4 & -1 & -1 & -1 & -1 \\
        1 & -2 & -1 & 3 & 1 & 0 & -3 \\
        1 & -2 & -1 & 1 & 0 & -3 & 3 \\
        1 & -2 & -1 & 0 & -3 & 3 & 1 \\
        1 & -2 & -1 & -3 & 3 & 1 & 0
        \end{array}\right] , 
        \left[\begin{array}{rrrrrrr}
        4 & 2 & 1 & 1 & 1 & 1 & 1 \\
        2 & -1 & 2 & -2 & -2 & -2 & -2 \\
        1 & 2 & -4 & -1 & -1 & -1 & -1 \\
        1 & -2 & -1 & 4 & -1 & -1 & -1 \\
        1 & -2 & -1 & -1 & 4 & -1 & -1 \\
        1 & -2 & -1 & -1 & -1 & 4 & -1 \\
        1 & -2 & -1 & -1 & -1 & -1 & 4
        \end{array}\right] , $}\\[0.3cm]
        \scalebox{0.45}[0.45]{$
        \left[\begin{array}{rrrrrrr}
        4 & 2 & 1 & 1 & 1 & 1 & 1 \\
        2 & -3 & 2 & 0 & 0 & -2 & -2 \\
        1 & 2 & 0 & -3 & -3 & -1 & -1 \\
        1 & 0 & -3 & 3 & -2 & -1 & -1 \\
        1 & 0 & -3 & -2 & 3 & -1 & -1 \\
        1 & -2 & -1 & -1 & -1 & 4 & -1 \\
        1 & -2 & -1 & -1 & -1 & -1 & 4
        \end{array}\right] , 
        \left[\begin{array}{rrrrrrr}
        4 & 2 & 1 & 1 & 1 & 1 & 1 \\
        2 & -3 & 2 & 0 & 0 & -2 & -2 \\
        1 & 2 & 0 & -3 & -3 & -1 & -1 \\
        1 & 0 & -3 & 2 & -1 & 1 & -3 \\
        1 & 0 & -3 & -1 & 2 & -3 & 1 \\
        1 & -2 & -1 & 1 & -3 & 0 & 3 \\
        1 & -2 & -1 & -3 & 1 & 3 & 0
        \end{array}\right] , 
        \left[\begin{array}{rrrrrrr}
        4 & 2 & 1 & 1 & 1 & 1 & 1 \\
        2 & -3 & 2 & 0 & 0 & -2 & -2 \\
        1 & 2 & -1 & -1 & -3 & 0 & -3 \\
        1 & 0 & -1 & -1 & -2 & -3 & 3 \\
        1 & 0 & -3 & -2 & 3 & -1 & -1 \\
        1 & -2 & 0 & -3 & -1 & 3 & 1 \\
        1 & -2 & -3 & 3 & -1 & 1 & 0
        \end{array}\right] , 
        \left[\begin{array}{rrrrrrr}
        3 & 3 & 2 & 1 & 1 & 1 & 0 \\
        3 & -1 & -2 & 1 & 0 & -3 & 1 \\
        2 & -2 & 1 & -2 & -2 & 2 & 2 \\
        1 & 1 & -2 & 0 & -3 & 1 & -3 \\
        1 & 0 & -2 & -3 & 3 & 1 & -1 \\
        1 & -3 & 2 & 1 & 1 & 0 & -3 \\
        0 & -1 & -2 & 3 & 1 & 3 & 1
        \end{array}\right] , $}\\[0.3cm]
        \scalebox{0.45}[0.45]{$
        \left[\begin{array}{rrrrrrr}
        3 & 3 & 2 & 1 & 1 & 1 & 0 \\
        3 & -1 & 0 & -1 & -2 & -3 & 1 \\
        2 & 0 & -3 & -2 & 0 & 2 & -2 \\
        1 & -1 & -2 & 1 & 3 & 0 & 3 \\
        1 & -2 & 0 & 3 & 1 & -1 & -3 \\
        1 & -3 & 2 & 0 & -1 & 3 & 1 \\
        0 & 1 & -2 & 3 & -3 & 1 & 1
        \end{array}\right] , 
        \left[\begin{array}{rrrrrrr}
        -1 & 1 & 3 & 1 & 0 & -3 & -2 \\
        -3 & -1 & 2 & 1 & -1 & 3 & 0 \\
        3 & 0 & 3 & 1 & 1 & 1 & 2 \\
        0 & -1 & 1 & -3 & 3 & 1 & -2 \\
        1 & -3 & 1 & -2 & -3 & -1 & 0 \\
        -1 & 3 & 1 & -3 & -1 & 0 & 2 \\
        2 & 2 & 0 & 0 & -2 & 2 & -3
        \end{array}\right] , 
        \left[\begin{array}{rrrrrrr}
        3 & 2 & 2 & 2 & 2 & 0 & 0 \\
        2 & 2 & -1 & -2 & -2 & 2 & 2 \\
        2 & -1 & 2 & -2 & -2 & -2 & -2 \\
        2 & -2 & -2 & 2 & -1 & 2 & -2 \\
        2 & -2 & -2 & -1 & 2 & -2 & 2 \\
        0 & 2 & -2 & 2 & -2 & -3 & 0 \\
        0 & 2 & -2 & -2 & 2 & 0 & -3
        \end{array}\right]$}\\[0.5cm]
\end{itemize}
Some of these matrices have nice structure, and we have reorganized some of them to show that structure. We did not fully analyze this. 
Below is the table giving all $44$ matrices. The notation in the left column, e.g. $A\oplus B\oplus C$ stands for taking all possible sums $A_i\oplus B_j\oplus C_k$. 
We write $mA$ for $A\oplus A\cdots\oplus A$ ($m$ times).
In the right column we write the multiplicity of this type in the list, due to different choices of indices. For example the entry $3A\oplus C$ consists of $2$ types: $3A\oplus C_1$ and $3A\oplus C_2$.
All multiplicities sum up to $44$, and it has been verified by the code-invariant that all of these are TH-inequivalent.
\begin{center}
    \begin{figure}[h]
        \caption{The full TH classification of $IW(7,25)$}
\begin{tabular}{|c|c|}
    \hline
    Type & Multiplicity\\
    \hline 
    $7A$ & 1\\
    \hline
    $5A\oplus B$ & 1\\
    \hline
    $3A\oplus 2B$ & 1\\
    \hline
    $A\oplus 3B$ & 1\\
    \hline
    $3A\oplus C$ & 2\\
    \hline
    $A\oplus B\oplus C$ & 2\\
    \hline
    $2A\oplus D$ &2\\ 
    \hline
    $B\oplus D$ & 2\\
    \hline
    $A\oplus E$ & 13\\
    \hline
    $F$ & 19\\
    \hline
\end{tabular}
\end{figure}
\end{center}

\bibliographystyle{alpha} 
\bibliography{weighing}

 \end{document}